\documentclass{article}

\usepackage[british]{babel}
\usepackage{amsmath}
\usepackage{amssymb}
\usepackage{amsthm}
\usepackage{verbatim}

\numberwithin{equation}{section}

\usepackage{hyperref}
\usepackage[framemethod=tikz]{mdframed}
\usepackage[titletoc]{appendix}

\usepackage{amsthm}

\newtheorem{defi}{Definition}[section]
\newtheorem{lem}[defi]{Lemma}
\newtheorem{rem}[defi]{Remark}

\newtheorem{theorem}[defi]{Theorem}

\newtheorem*{theorem*}{Theorem}
\newtheorem*{asum*}{Assumption}

\begin{document}

\title{From Hawkes-type processes to stochastic volatility}
\author{\L ukasz Treszczotko\footnote{Institute of Mathematics, University of Warsaw, Banacha 2 02-097 Warsaw}\\
\href{mailto:lukasz.treszczotko@gmail.com}{lukasz.treszczotko@gmail.com} }

\maketitle

\begin{abstract}
We introduce a Hawkes-like process and study its scaling limit as the system becomes increasingly endogenous. We derive functional limit theorems for intensity and fluctuations. Then, we introduce a high-frequency model for a price of a liquid traded financial instrument in which the nearly unstable regime leads to a Heston-type process where the negative correlation between the noise driving the proce of the instrument and the volatility can be viewed as a result of high variance of the sell-side order arrivals. 
\end{abstract}

{\bf Keywords:} 
\\

\textup{2010} \textit{Mathematics Subject Classification}: \textup{Primary: 60G55, 60F17}

\maketitle

\section{Introduction}
Given their self-exciting nature, Hawkes process have been widely used to describe the behaviour of prices of financial instruments. Lately they have been shown to provide micro-structural foundations for the Heston (see~\cite{JR1}) and rough Heston models (see~\cite{JR2} and~\cite{OMAR}) by making the market highly endogenous. In this paper we first study the nearly unstable regime for a Hawkes-type process and provide functional limit theorems regarding limit intensity and large scale fluctuations (Theorems~\ref{THM1} and~\ref{THM2}). Then we apply these processes to provide a micro-structural model for price fluctuations and volatility (Theorems~\ref{THM3} and~\ref{THM4}). We obtain a Heston type model which can be specified as follows
\begin{eqnarray*}\label{heston}
X_t^+&=& a\int_0^t(x_0^+ - X_s^+)ds + b\int_0^t\sqrt{X_s^+}dW_r^1\nonumber + b\int_0^t\sqrt{X_s^-}dW_r^2,\nonumber  \\
X_t^-&=& a\int_0^t(x_0^- - X_s^-)ds + b\int_0^t\sqrt{X_s^+}dW_r^3\nonumber + b\int_0^t\sqrt{X_s^-}dW_r^4,\nonumber  \\
P_t&=& \int_0^t \sqrt{X_r^+}d\Big(W^{1}_r-W^{3}_r\Big) + \int_0^t \sqrt{X_r^-}d\Big(W^{2}_r-W^{4}_r\Big),
\end{eqnarray*}
Where $a,b$ are positive constants, $(X_t^+)$, $(X_t^-)$ can be viewed as bid and side volatility, respectively, $(P_t)$ is the price of the underlying financial instrument and $(W^i)_t$ are independent Brownian motions which, however, do depend on the volatilities. In other words the above is a weak formulation. Our model provides an intuitive explanation for the stylized fact which says that the noise driving the price should be negatively correlated with the noise driving the volatility. It does so by saying that this negative dependence is due to higher variance of the sell-side order arrivals.

In Section~\ref{chap2} we introduce the Hawkes type processes, state two theorems describing their nearly unstable limit and prove them. Building on that, in Section~\ref{chap3} we introduce a tick-by-tick price in which bid and ask orders have random intensity and duration of execution which is very appealing intuitively.  We derive nearly unstable limits which give the Heston-type model described above. Crucially we assume that the intensity and time of execution have finite variance. We strongly suspect that if we assume heavy-tailed distribution for the random time of execution we are bound to obtain some type of rough volatility model. However, this remains to be seen.

Our work is inspired by the results in~\cite{JR1} and~\cite{JR2} and the models studied therein can be seen as mean-field approximations to the model provided by us.

\subsection{Notation}
By $\mathcal{D}_E[0,\tau]$ we denote the Skorohod space of cadlag processes with values in a separable metric space $E$ with $J_1$ topology.

\section{Nearly unstable Hawkes-type processes}\label{chap2}

Consider the following population model. The total population size is governed by a Poisson process with a random time-dependent intensity $(\lambda_t)_{t\geq 0}$  such that for a given $t>0$, $\lambda_t$ may depend on the whole history of the population growth until time $t$. This dependence can be described as follows: to every reproduction event (an arrival event of a Cox process) we attach a random variable $X\geq 0$ which tells us how much offspring was produced in this event, and a random variable $Y\geq 0$ gives the \emph{reproduction-lifetime} of each child connected to this reproduction event. During its reproduction-lifetime each child increases the intensity $\lambda_t$ by one. After this time it no longer does so. We also assume an exogenous immigration rate equal to $\lambda_0$

The above description leads to the following equation that $(\lambda_t)$ must satisfy  
\begin{equation}\label{lambda_1}
\lambda_t = \lambda_0 + \int_0^t\int_0^{\lambda_{s-}}\int_0^\infty\int_0^\infty x\mathbf{1}_{\{t-s<y\}}M(ds,du,dx,dy),
\end{equation}
where $M$ is a Poisson random measure on $\mathbb{R}_+^4$ independent of $(\lambda_t)$ with intensity given by
\begin{equation}
ds\otimes du \otimes \mathbb{P}_X(dx) \otimes \mathbb{P}_Y(dy),
\end{equation}
with $\mathbb{P}_X(dx)$ and $\mathbb{P}_Y(dy)$ being distributions of $X$ and $Y$, respectively. By $N_t$ the process such that $N_t$ is the total number of offspring at time $t\geq 0$. It is not hard to show (for example by computing $\mathbb{E}(\lambda_t)$) that, as long as $\mathbb{E}(X)\mathbb{E}(Y)<1$, we have no explosion almost surely, i.e., the number of arrivals of $(\lambda_t)$ is alsmost surely finite on any compact interval of $\mathbb{R}_+$.

\subsection{Nearly unstable regime: intensity}\label{chap2_1}

Our aim is to study the nearly unstable regime in which $\mathbb{E}(X)\mathbb{E}(Y)$ gets close to one. To do so we rewrite~\eqref{lambda_1} in the following way
\begin{equation}\label{lambda_2}
\lambda_t^T = \lambda_0^T + \int_0^t\int_0^{\lambda_{s-}^T}\int_0^\infty\int_0^\infty a_Tx\mathbf{1}_{\{t-s<y\}}M(ds,du,dx,dy),
\end{equation}
where $a^T \in (0,1)$ and $T>0$. Following~\cite{JR1}, we take $a_T\rightarrow 1$ as $T\rightarrow \infty$. To obtain a non-trivial limit we rescale time and consider 
\begin{equation}\label{l_tilde}
(\widetilde{\lambda_t^T})_{t\geq 0}:=\left(\frac{\lambda_{Tt}}{F_T}\right)
\end{equation}
for a suitable deterministic normalization $F_T$. But first let us rewrite~\eqref{lambda_2} in the following way
\begin{eqnarray}\label{lambda_3}
\lambda_t^T &=&\lambda_0^T + \int_0^t\int_0^{\lambda_{s-}^T}\int_0^\infty\int_0^\infty a_Tx\mathbf{1}_{\{t-s<y\}}\widetilde{M}(ds,du,dx,dy)\\
&& \: + \int_0^t \lambda_{s-}^T a_T \mathbb{E}(X)\mathbb{P}(Y>t-s)ds,
\end{eqnarray}
where $\widetilde{M}$ is the compensated Poisson random measure. Notice that in the second line of~\eqref{lambda_3} we can replace $\lambda_{s-}^T$ by $\lambda_{s}^T$. It will be convenient to introduce some notation. Let us put
\begin{equation}
\Phi^T(z):=a_T\mathbb{P}(Y>z), \quad z\geq 0
\end{equation}
and
\begin{equation}\label{m_def}
m:=\int_0^\infty z\mathbb{P}(Y>z)z=\frac{1}{2}\mathbb{E}(Y^2).
\end{equation}
For the rest of Section~\ref{chap2} we make the following assumption.
\begin{asum*}[\textbf{A}] 
Assume that $\mathbb{E}(X)=1=\mathbb{E}(Y)$ and, moreover, both $\mathbb{E}(X^2)$ and $\mathbb{E}(Y^2)$ are finite.
\end{asum*}
Recall the following classical lemma.
\begin{lem}\label{lem_1}
If $f(t)=h(t)+\int_0^t\phi(t-s)f(s)ds$ with $h$ measurable and locally bounded and $\phi \in L^1(0,\infty)$ with $\int_0^\infty\phi(z)dz<1$, then 
\begin{equation}
f(t) = h(t)+\int_0^t\phi(t-s)h(s)ds,
\end{equation}
where $\psi(z)=\sum_{k=1}^\infty \big(\phi\big)^{\ast k}(z)$ is the infinite sum of convolutions.
\end{lem}

Using Lemma~\ref{lem_1} with 
\begin{eqnarray}
h(s)&=&\lambda_0^T + \int_0^t\int_0^{\lambda_{s-}^T}\int_0^\infty\int_0^\infty a_Tx\mathbf{1}_{t-s<y}\widetilde{M}(ds,du,dx,dy),\\
\phi(z)&=& a_T \mathbb{P}(Y>z)=:\Phi^T(z)
\end{eqnarray}
we get
\begin{eqnarray}\label{lambda_4}
\lambda_t^T &=&\lambda_0^T + \lambda_0^T\Psi^T(t-s)ds \\
&& \: + \int_0^t\int_0^{\lambda_{s-}^T}\int_0^\infty\int_0^\infty a_Tx\mathbf{1}_{\{t-s<y\}}\widetilde{M}(ds,du,dx,dy)\\
&& \: + \int_0^t\Psi^T(t-s)\Big(\int_0^s\int_0^{\lambda_{r-}^T}\int_0^\infty\int_0^\infty a_Tx\mathbf{1}_{s-r<y}\\
&& \:\widetilde{M}(ds,du,dx,dy)\Big)ds,
\end{eqnarray}
where $\Psi^T(z)=\sum_{k=1}^\infty \big(\Phi^T\big)^{\ast k}(z)$. Using Proposition 2.1 in~\cite{JR1} and the fact that the functions $\Psi^T$ are non-increasing one can easily show that the following lemma holds.

\begin{lem}\label{lem_T0}
Let 
\begin{equation}
\rho^T(z):=\frac{T(1-a_T)}{a_T}\Psi^T(Tz), \quad z\geq 0.
\end{equation}
Then $\rho^T$ is non-increasing and $\int_0^\infty \rho^T(z)dz = 1$ for all $T>0$. Let $Z^T$ be a non-negative random variable with density $\rho^T$. We have the following:
\begin{itemize}
\item[(i)]
$Z^T$ converges weakly to an exponential random variable with density $\frac{\lambda}{m}\exp(-\frac{\lambda x}{m})dx$.
\item[(ii)]
$\rho^T$ converges to $\frac{\lambda}{m}\exp(-\frac{\lambda z}{m})$ uniformly on $[0,\infty)$.
\end{itemize}
\end{lem}

Changing the order of integration we have
\begin{eqnarray}\label{lambda_5}
\lambda_t^T &=&\lambda_0^T + \lambda_0^T\Psi^T(t-s)ds \\
&& \: + \int_0^t\int_0^{\lambda_{r-}^T}\int_0^\infty\int_0^\infty a_Tx\Bigg(\mathbf{1}_{t-r<y}\\
&& \: + \int_0^{t-r}\Psi^T(t-r-s)\mathbf{1}_{\{s<y\}}ds\Bigg)\widetilde{M}(ds,du,dx,dy).
\end{eqnarray}
At the same time, the total number of offspring born until $t$ is equal to
\begin{eqnarray}\label{N_1}
N_t^T&=& \int_0^t\int_0^{\lambda_{r-}^T}\int_0^\infty\int_0^\infty x\widetilde{M}(ds,du,dx,dy) \\
&& \:+ \int_0^t\lambda_{r}^Tdr.
\end{eqnarray}.

It turns out that the suitable scaling in this case is given by $F_T=T$ and we may rewrite~\eqref{lambda_5} as
\begin{eqnarray}\label{lambda_6}
\frac{\lambda_{Tt}^T}{T} &=&\frac{\lambda_0^T}{T} + \lambda_0^T\int_0^t \Psi^T(T(t-s))ds \\
&& \: + \int_0^t\int_0^{\lambda_{Tr-}^T/T}\int_0^\infty\int_0^\infty a_Tx\Bigg(\frac{1}{T}\mathbf{1}_{\{T(t-r)<y\}}\nonumber\\
&& \: + \int_0^{t-r}\Psi^T(T(t-r-s))\mathbf{1}_{\{Ts<y\}}ds\Bigg)\nonumber\\
&& \:\widetilde{M}(Tdr,Tdu,dx,dy),\nonumber
\end{eqnarray}
where $\widetilde{M}(Tds,Tdu,dx,dy)$ is a compensated Poisson random measure with intensity
\begin{equation}
Tds\otimes Tdu \otimes \mathbb{P}_X(dx) \otimes \mathbb{P}_Y(dy).
\end{equation}

We have the following theorem which says that the limit intensity process follows a CIR-type process just as in Theorem 2.2 in~\cite{JR1}. This suggests that the result obtained by Jaisson and Rosenbaum is much more universal and can be seen in a way which resembles high-frequency markets a little more intuitively.  

\begin{theorem}\label{THM1}
Fix any $\tau>0$. Take $a_T=1-\lambda/T$ for some fixed $\lambda>0$ and fix $\lambda_0^T=\lambda_0>0$. Then the process 
\begin{equation}
\widetilde{\lambda}^T_t:=\frac{\lambda_{Tt}^T}{T}, \quad t\geq 0,
\end{equation}
converges weakly in $\mathcal{D}[0,\tau]$ and the limit process $(\lambda_t)$ satisfies
\begin{equation}\label{L_lim}
\widetilde{\lambda}_t = \int_0^t \theta(t-r)dr + \sigma_1 \sigma_2 \int_0^t \theta(t-r)\sqrt{\widetilde{\lambda}_r}dW_r, \quad t\in[0,\tau],
\end{equation}
where $\sigma_1^2=\mathbb{E}(X^2)$, $\sigma_2^2=\mathbb{E}(Y^2)$, $\theta(z)=1/m\exp(-\lambda z/m)$ and $(W_t)$ is a standard Brownian motion.
\end{theorem}

\begin{proof}
We may rewrite~\eqref{lambda_6} again as follows:
\begin{eqnarray}\label{lambda_7}
\widetilde{\lambda}^T_t &=&R^T_t+\lambda_0\int_0^t\frac{1}{m}\exp\big(-\frac{\lambda}{m}(t-s)\big)ds \\
&& \: + \frac{1}{T}\int_0^t\int_0^{\widetilde{\lambda}^T_{r-}}\int_0^\infty\int_0^\infty xy\frac{1}{m}\exp\big(-\frac{\lambda}{m}(t-r)\big)\nonumber\\
&& \:\widetilde{M}(Tdr,Tdu,dx,dy),\nonumber
\end{eqnarray}
where $R^T_t=\lambda_0/T+R^{T,1}_t+R^{T,2}_t+R^{T,3}_t$ with
\begin{eqnarray*}
R^{T,1}_t &:=& \lambda_0^T\int_0^t \Psi^T(T(t-s))ds - \lambda_0\int_0^t\frac{1}{m}\exp\big(-\frac{\lambda}{m}(t-s)\big)ds,\\
R^{T,2}_t &:=& \int_0^t\int_0^{\widetilde{\lambda}^T_{r-}}\int_0^\infty\int_0^\infty \frac{1}{T}a_Tx\mathbf{1}_{\{T(t-r)<y\}}\widetilde{M}(Tds,Tdu,dx,dy),\\
R^{T,3}_t &:=& \int_0^t\int_0^{\widetilde{\lambda}^T_{r-}}\int_0^\infty\int_0^\infty a_Tx\Bigg(\int_0^{t-r}\Psi^T(T(t-r-s))\mathbf{1}_{\{Ts<y\}}ds\\
&& \: - \frac{y}{T}\frac{1}{m}\exp\big(-\frac{\lambda}{m}(t-r)\big)\Bigg)\widetilde{M}(Tds,Tdu,dx,dy).
\end{eqnarray*}
We need the following lemma which we prove in the appendix.

\begin{lem}\label{r_lem}
For any $\tau>0$, $(R^T_T)_{0\leq t \leq \tau}$ converges to $0$ in $\mathcal{D}([0,\tau])$ as $T\rightarrow \infty$.
\end{lem}

Let $K(T)$ be positive that $K(T)\rightarrow \infty$ as $T\rightarrow \infty$. The exact choice will be specified later. Notice that equation~\eqref{lambda_7} can be rewritten as
\begin{equation}
\widetilde{\lambda}^T_t =R^T_t + U_t^T + A_t^{T,1} + A_t^{T,2},
\end{equation}
where
\begin{eqnarray}
U_t^T &=& \lambda_0\int_0^t\frac{1}{m}\exp\big(-\frac{\lambda}{m}(t-s)\big)ds, \\
A_t^{T,1} &=& \frac{1}{T}\int_0^t\int_0^{\widetilde{\lambda}^T_{r-}}\int_{[0,K(T)]} \int_{[0,K(T)]} xy\frac{1}{m}\exp\big(-\frac{\lambda}{m}(t-r)\big)\\ && \:\widetilde{M}(Tdr,Tdu,dx,dy),\label{loc_12} \nonumber\\
A_t^{T,2} &=& \frac{1}{T}\int_0^t\int_0^{\widetilde{\lambda}^T_{r-}}\int_{(K(T), \infty]} \int_{(K(T), \infty]} xy\frac{1}{m}\exp\big(-\frac{\lambda}{m}(t-r)\big)\\
A_t^{T,3} &=& \frac{1}{T}\int_0^t\int_0^{\widetilde{\lambda}^T_{r-}}\int_{[0,K(T)]} \int_{(K(T), \infty]} xy\frac{1}{m}\exp\big(-\frac{\lambda}{m}(t-r)\big)\\
A_t^{T,4} &=& \frac{1}{T}\int_0^t\int_0^{\widetilde{\lambda}^T_{r-}}\int_{(K(T), \infty]} \int_{[0,K(T)]} xy\frac{1}{m}\exp\big(-\frac{\lambda}{m}(t-r)\big)\\
&& \: \widetilde{M}(Tdr,Tdu,dx,dy). \nonumber
\end{eqnarray}
It is easy to show that by Doob's $L^2$ inequality the sequence of processes $(A_t^{T,2})$ converges to zero in $\mathcal{D}[0,\tau]$. Indeed,
\begin{equation}
\mathbb{P}(\sup_{t\in [0, \tau]}|A_t^{T,2}|>\epsilon)\leq \frac{4}{\epsilon^2}\mathbb{E}(A_\tau^{T,2})^2,
\end{equation}
with 
\begin{equation}\label{loc_11}
\mathbb{E}(A_t^{T,2})^2 = \int_0^\tau \mathbb{E}(\widetilde{\lambda}_r^T)\mathbb{E}\big(X^2\mathbf{1}_{\{X>K(T)\}}\big)\mathbb{E}\big(Y^2\mathbf{1}_{\{Y>K(T)\}}\big) dr.
\end{equation}
Similarly $(A_t^{T,3})$ and $(A_t^{T,4})$ converge to $0$ in $\mathcal{D}[0,\tau]$. Since $\mathbb{E}(\widetilde{\lambda}_r^T)$ is bounded for all $r\geq 0$ and $T>0$,  we see that the right-hand side of~\eqref{loc_11} converges to $0$ as $T\rightarrow \infty$. Now comes the crucial part. Note that the right-hand side of~\eqref{loc_12} can be rewritten as (we put $\theta(z):= 1/m \exp(-(\lambda z/m))$ for greater clarity)
\begin{equation}
\int_0^t \sqrt{\widetilde{\lambda}_r^T}\theta(t-r)dW^T_r,
\end{equation}
where for $r \geq 0$
\begin{eqnarray}\label{W_def}
W^T_r&:=& \frac{1}{T}\int_0^r\int_0^{\widetilde{\lambda}^T_{s-}}\int_{[0,K(T)]} \int_{[0,K(T)]} \\
&& \: \frac{1}{\sqrt{\widetilde{\lambda}_s^T}}xy\widetilde{M}(Tds,Tdu,dx,dy).
\end{eqnarray}
Now we would like to show that the sequence of processes $(W^T_t)$ converges to a Brownian motion as $T\rightarrow \infty$.

\begin{lem}\label{BM_lem}
As $T\rightarrow \infty$ $(W^T_t)$ converges weakly in $\mathcal{D}[0,\tau]$ to $(\sigma_1 \sigma_2 W_t)$, where $\sigma_1^2=\mathbb{E}(X^2)$, $\sigma_2^2=\mathbb{E}(Y^2)$ and $(W_t)$ is the standard Brownian motion.
\end{lem}

The proof is given in the appendix.

Now we are well equipped to prove the desired convergence. By Theorem 5.4 in~\cite{PROTTER}
\begin{equation}
\Big(R^T_t + U_t^T + A_t^{T,1}, \widetilde{\lambda}^T_t, W_t^t\Big)_{t\geq 0}
\end{equation}
is weakly relatively compact in $\mathcal{D}_{\mathbb{R}^3}[0,\tau]$ and since the limit SDE
\begin{equation}\label{L_lim_plim}
\widetilde{\lambda}_t = \int_0^t \theta(t-r)dr + \sigma_1 \sigma_2 \int_0^t \theta(t-r)\sqrt{\widetilde{\lambda}_r}dW_r
\end{equation}
has a global solution and weak local uniqueness holds, we coclude that $(\widetilde{\lambda}_t^T)$ converges weakly in $\mathcal{D}[0,\tau]$ to $(\widetilde{\lambda}_t)$ for any $\tau>0$.

\end{proof}

\newpage

\subsection{Nearly unstable regime: total offspring distribution}\label{chap2_2}

Recall that by  $N^T_t$ we denote the total offspring at time $t>0$, i.e., 
\begin{equation}
N_t^T = \int_0^t\int_0^{\lambda^T_{r-}}\int_0^\infty\int_0^\infty x  M(dr,du,dx,dy),
\end{equation}
which equals
\begin{eqnarray}
N_t^T &=& \int_0^t\int_0^{\lambda^T_{r-}}\int_0^\infty\int_0^\infty x  \widetilde{M}(dr,du,dx,dy) \\
&& \: + \int_0^t \lambda^T_{r-} dr
\end{eqnarray}
After rescaling we have
\begin{equation}\label{Z_def}
Z^T_t:=\frac{1}{T}\left(\widetilde{N}_{Tt}^T - \int_0^{Tt} \widetilde{\lambda}^T_{r-}dr\right) = \int_0^t\sqrt{\widetilde{\lambda}^T_{r-}}\big(dB^T_r+dE^T_r\big),
\end{equation}
where 
\begin{equation}\label{B_def}
B^T_r:= \int_0^r\int_0^{\widetilde{\lambda}^T_{s-}}\int_0^{K(T)}\int_0^\infty x \frac{1}{\sqrt{\widetilde{\lambda}^T_{s-}}} \widetilde{M}(Tds,Tdu,dx,dy),
\end{equation}
\begin{equation}\label{E_def}
E^T_r:= \int_0^r\int_0^{\widetilde{\lambda}^T_{s-}}\int_{K(T)}^\infty\int_0^\infty x \frac{1}{\sqrt{\widetilde{\lambda}^T_{s-}}} \widetilde{M}(Tds,Tdu,dx,dy),
\end{equation}
and $K(T)$ is as before.

\begin{theorem}\label{THM2}
As $T\rightarrow \infty$ the process $Z^T$ defined by~\eqref{Z_def} converges weakly in $\mathcal{D}[0,\tau]$ to 
\begin{equation}\label{Z_lim}
(\sigma_1 \int_0^t \sqrt{\widetilde{\lambda}_s}dB_s)_{t\geq 0},
\end{equation}
where $B$ is a standard Brownian motion, $(\sqrt{\widetilde{\lambda}_s})_{s\geq 0}$ satisfies~\eqref{L_lim_plim} and $\sigma_1=\sqrt{\mathbb{E}(X^2)}$.
\end{theorem}

\begin{proof}
In the same manner as in Section~\ref{chap2_1} one can show that, as $T\rightarrow \infty$, $(E^T_r)$ and $(B^T_r)$ converge in $\mathcal{D}[0,\tau]$ to $0$ and $(\sigma_1 B_r)$ respectively, where $B$ is a standard Brownian motion. In fact the pair $(W_t^T, B_t^T)$ converges weakly in $\mathcal{D}[0,\tau]$ to $\sigma_1 \sigma_2 W_t, \sigma_1 B_t)$ (see for Theorem 2.1 in~\cite{WHITT}), where $W$ and $B$ is a pair of correlated Brownian motions with correlation coefficient $\rho=\mathbb{E}(Y)/\sqrt{\mathbb{E}(Y^2)}$. In order to prove the desired convergence one only has to use Theorem 5.4 in~\cite{PROTTER} again.

\end{proof}

\newpage

\section{Tick by tick price model}\label{chap3}

\subsection{Model description}
Following a growing number of papers (see for instance~\cite{JR1}, ~\cite{OMAR}) in which Hawkes processes are applied in modelling market prices of financial instruments we consider a tick by tick price model for one traded financial instrument. Imagine that we have only two types of orders: \emph{bid} (buy a specified amount of the instrument) and \emph{ask} (sell a specified amount of the instrument). To make our model resemble the actual trading in electronic markets we assume that every order is specified by a tuple $(x,y)$ where $x$ is the instantaneous intensity of the order execution and $y$ is the time it takes to complete the order. We implicitly assume that the market is liquid enough for us to be able to execute the order without any delays. Thus the total size of the order is given by $x$ times $y$. To take into account the endogenous nature of many orders in the modern electronic trading markets we assume that the bid/sell intensities depend on the current volatility of the instrument. To be more precise let the bid intensity be given by
\begin{eqnarray}\label{l+}
\lambda_t^+ &=& \lambda_0^+ + a_1\int_0^t\int_0^{\lambda_{s-}^+}\int_{[0,\infty)^2}x\mathbf{1}_{\{t-s<y\}}M_1(ds,du,dx,dy)\nonumber\\
&& \: + a_2\int_0^t\int_0^{\lambda_{s-}^-}\int_{[0,\infty)^2}x\mathbf{1}_{\{t-s<y\}}M_2(ds,du,dx,dy), \quad t>0.
\end{eqnarray}
There is a number of things to explain so let us start. Firstly, $(\lambda_t^+)_{t\geq 0}$ is the bidding intensity process, $a_1,a_2>0$ and $\lambda_0^+$ is the baseline exogenous bid intensity which is constant. Secondly, $M_1$ is a Poisson random measure on $[0,\infty)^4$ with intensity 
\begin{equation}\label{mean}
ds\otimes du\otimes \mathbb{P}_{X_1}(dx)\otimes \mathbb{P}_{Y_1}(dy),
\end{equation}
where $\mathbb{P}_{X_1}$ and $\mathbb{P}_{Y_1}$ are probability distributions on $[0,\infty)$.  $({\lambda_{s-}^-})_{s\geq 0}$ is the ask intensity process with $M_2$ being another Poisson random measure on $[0,\infty)^4$, independent of $M_1$, with intensity as in~\eqref{mean} but perhaps with a different distribution of the $(x,y)$ marks. This seems relatively complicated at first sight but its interpretation is actually very intuitive. Indeed, $\lambda_t^+$ is proportional to the current number of active bid orders, taking into account their size but not the duration of their execution, which is known only to the actual agents sending the orders anyway. The two random measures $M_1$ and $M_2$ represent the arrivals of the orders on the market with $M_1$ being bid orders responding to other bid orders and $M_2$ being bid orders responding to ask orders. To complete the picture we have
\begin{eqnarray}\label{l-}
\lambda_t^- &=& \lambda_0^+ + a_3\int_0^t\int_0^{\lambda_{s-}^+}\int_{[0,\infty)^2}x\mathbf{1}_{\{t-s<y\}}M_3(ds,du,dx,dy)\nonumber\\
&& \: + a_4\int_0^t\int_0^{\lambda_{s-}^-}\int_{[0,\infty)^2}x\mathbf{1}_{\{t-s<y\}}M_3(ds,du,dx,dy), \quad t>0,
\end{eqnarray}
which describes the evolution of ask intensity process. The notation is self-explanatory. We assume that the price at time $t\geq 0$ of the instrument under consideration is proportional to the difference between the total size of bid orders and the total size of ask orders until time $t$, that is
\begin{equation}\label{price}
P_t=N_t^+ - N_t^-,
\end{equation}
where
\begin{eqnarray}
N_t^+ &=& \int_0^t\int_0^{\lambda_{r-}^+}\int_{[0,\infty)} x\big(y\wedge (t-r)\big) M_1(dr,du,dx,dy)\nonumber\\
&& \: + \int_0^t\int_0^{\lambda_{r-}^-}\int_{[0,\infty)} x\big(y\wedge (t-r)\big) M_2(dr,du,dx,dy),\label{n_plus}
\end{eqnarray}
and
\begin{eqnarray}
N_t^- &=& \int_0^t\int_0^{\lambda_{r-}^-}\int_{[0,\infty)} x\big(y\wedge (t-r)\big) M_4(dr,du,dx,dy)\nonumber\\
&& \: + \int_0^t\int_0^{\lambda_{r-}^+}\int_{[0,\infty)} x\big(y\wedge (t-r)\big) M_3(dr,du,dx,dy).\label{n_minus}
\end{eqnarray}
In a reasonable market we expect the price to be a martingale. Notice that compensating the Poisson random measures $M_1,\ldots,M_4$, the price process $(P_t)_{t\geq 0}$ is a martingale as long for any $t>0$ (see Theorem 8.23 in~\cite{PZ})
\begin{eqnarray*}\label{mart_con}
+ &&\int_0^t \lambda_{s}^+ \mathbb{E}(X_1)\mathbb{E}\Big(\big(Y_1\wedge (t-s)\big)\Big)ds \\
+ && \int_0^t \lambda_{s}^- \mathbb{E}(X_2)\mathbb{E}\Big(\big(Y_2\wedge (t-s)\big)\Big)ds \\
- && \int_0^t \lambda_{s}^- \mathbb{E}(X_4)\mathbb{E}\Big(\big(Y_4\wedge (t-s)\big)\Big)ds \\
- && \int_0^t \lambda_{s}^+ \mathbb{E}(X_3)\mathbb{E}\Big(\big(Y_3\wedge (t-s)\big)\Big)ds \\
\end{eqnarray*}
is equal to $0$. Thus, it suffices to assume that $\mathbb{E}(X_1)=\mathbb{E}(X_3)$ and $\mathbb{E}(X_2)=\mathbb{E}(X_4)$ and that $Y_1,\ldots,Y_4$ are equal in law. In fact, for further tractability of the model we assume a little more. 
\begin{asum*}[\textbf{B}]\label{asB}
Assume that $\mathbb{E}(X_1)=\ldots =\mathbb{E}(X_4)$, $Y_1,\ldots,Y_4$ have the same distribution and that $X_1,\ldots,X_4,Y_1$ have finite variance.
\end{asum*}
We also make the following, somewhat technical assumption.
\begin{asum*}[\textbf{C}]\label{asC}
Assume that for any $i=1,\ldots,4$ $Y_i$ has a regularly varying density at infinity with index $-3-\gamma$ for some $\gamma>0$.
\end{asum*}
Finally, to make sure that the Hawkes-type processes do not explode in finite time we have to assume that for $i=1,\ldots,4$
\begin{equation}\label{radius}
s(A)\int_0^\infty \mathbb{P}(Y_1>s)ds<1
\end{equation}
where $s(A)$ is the spectral radius of the matrix
\begin{equation}\label{smatrix}
\begin{bmatrix}
    a_1\mathbb{E}(X_1)       & a_2\mathbb{E}(X_2) \\
    a_3\mathbb{E}(X_2)       & a_4\mathbb{E}(X_4)
\end{bmatrix}
.
\end{equation}

\subsection{Nearly unstable limit}

Now we would like to obtain a functional limit theorem for the model constructed in the previous section. Let $T>0$.  Following~\cite{JR1}, we will investigate the case when the quantity in~\eqref{radius} approaches $1$, i.e., the situation in which the market becomes increasingly endogenous. Assume first the $a_i$ coefficients in the matrix~\eqref{smatrix} are all equal to $a_T \in (0,1)$ and that 
\[\int_0^\infty \mathbb{P}(Y_1>s)ds = \mathbb{E}(Y_1)=1.\]
To study the nearly unstable limit we will rescale time by $T$ and and  at the same time take $a_T\rightarrow 1$ as $T\rightarrow\infty$. This approach was pioneered by Jaisson and Rosenbaum in~\cite{JR1} and later used by the same authors in~\cite{JR2}. 

Now we proceed very similarly as in Section~\ref{chap2_1}. Suppose that the assumptions (\textbf{B}) and (\textbf{C}) hold and, moreover, $\mathbb{E}(X_1)=1/2$, $\mathbb{E}(Y_1)=1$. Let $a_T$ be in $(0,1)$ such that 
\[ \lim_{T \rightarrow \infty}T(1-a_T)=\lambda,\]
for some $\lambda>0$ and put
\[m:=\frac{1}{2}\mathbb{E}(Y_1^2).\]

\subsubsection{Volatility processes}
In this section we find the scaling limits of the intensity processes $(\lambda_t^+)_{t\geq 0}$ and $(\lambda_t^-)_{t\geq 0}$ given by~\eqref{l+} and~\eqref{l-}, respectively. To formulate the theorem we need to introduce some additional notation. Define
\begin{equation}
\phi^T(z):=a_T\mathbb{P}(Y_1>z), \quad z\geq 0,
\end{equation}
and put 
\[\lambda_0^{T}:=\lambda_0^{T,+}+\lambda_0^{T,-}.\]
We may rewrite~\eqref{l+} as
\begin{eqnarray}\label{l++}
\lambda_t^{T,+}&:=&\lambda_0^{T,+} + \int_0^t \lambda_s^{T,+} \phi^T(t-s)ds\nonumber\\
&& \: + a_T\int_0^t\int_0^{\lambda_{s-}^{T,+}}\int_{[0,\infty)^2}x\mathbf{1}_{\{t-s<y\}}\widetilde{M_1}(ds,du,dx,dy)\nonumber\\
&& \: + a_T\int_0^t\int_0^{\lambda_{s-}^{T,-}}\int_{[0,\infty)^2}x\mathbf{1}_{\{t-s<y\}}\widetilde{M_2}(ds,du,dx,dy),
\end{eqnarray}
and similarly
\begin{eqnarray}\label{l--}
\lambda_t^{T,-}&=&\lambda_0^{T,-} + \int_0^t \lambda_s^{T,-} \phi^T(t-s)ds\nonumber\\
&& \: + a_T\int_0^t\int_0^{\lambda_{s-}^{T,+}}\int_{[0,\infty)^2}x\mathbf{1}_{\{t-s<y\}}\widetilde{M_3}(ds,du,dx,dy)\nonumber\\
&& \: + a_T\int_0^t\int_0^{\lambda_{s-}^{T,-}}\int_{[0,\infty)^2}x\mathbf{1}_{\{t-s<y\}}\widetilde{M_4}(ds,du,dx,dy).
\end{eqnarray}

The rescaled volatility processes are defined by
\begin{eqnarray}\label{rescaled_vol}
\widetilde{\lambda_t^{T,+}}&:=&\frac{1}{T}\lambda_{Tt}^{T,+},\\
\widetilde{\lambda_t^{T,-}}&:=&\frac{1}{T}\lambda_{Tt}^{T,-}.
\end{eqnarray}
Assume now that $\lambda_0^{T,+}=\lambda_0^+>0$, $\lambda_0^{T,-}=\lambda_0^->0$ for all $T>0$. We are ready to formulate the main theorem of this section.

\begin{theorem}\label{THM3}
For any $\tau>0$ the process
\begin{equation}\label{6-form}
\Bigg(\widetilde{\lambda_t^{T,+}}, \widetilde{\lambda_t^{T,-}}, W_t^{T,1}, W_t^{T,2}, W_t^{T,3}, W_t^{T,4}\Bigg)_{t\geq 0}
\end{equation}
is tight in $\mathcal{D}_{\mathbb{R_+}^2\times \mathbb{R}^4}[0,\tau]$ and converges weakly as $T\rightarrow \infty$ to
\begin{equation}\label{6-limit}
\Bigg(\widetilde{\lambda_t^+}, \widetilde{\lambda_t^-}, W_t^1, W_t^2, W_t^3, W_t^4\Bigg)_{t\geq 0}
\end{equation}
where
\begin{eqnarray}\label{log8}
\widetilde{\lambda_t^+}&=& \lambda_0^+\int_0^t\theta(t-s)ds + \int_0^t\theta(t-r)\sqrt{\widetilde{\lambda_{r-}^+}}dW_r^1\nonumber + \int_0^t\theta(t-r)\sqrt{\widetilde{\lambda_{r-}^-}}dW_r^2 \nonumber \\
\widetilde{\lambda_t^-}&=& \lambda_0^-\int_0^t\theta(t-s)ds + \int_0^t\theta(t-r)\sqrt{\widetilde{\lambda_{r-}^+}}dW_r^3\nonumber + \int_0^t\theta(t-r)\sqrt{\widetilde{\lambda_{r-}^-}}dW_r^4, \nonumber \\
\end{eqnarray}
and $\big(W_t^1, W_t^2, W_t^3, W_t^4\big)_{t\geq 0}$ is a four-dimensional Brownian motion with diagonal covariance matrix given by~\eqref{cov}.
\end{theorem}

\begin{rem}
In fact the unique global solution to\eqref{log8} is continuous and if we put $X_t^+:=\lambda \lambda_t^+$, $X_t^-:=\lambda \lambda_t^-$, $x_0^+:=\lambda \lambda_0^+$ and $x_0^-:=\lambda \lambda_0^-$, then by applying integration by parts one may show that these processes satisfy
\begin{eqnarray}\label{log9}
X_t^+&=& \int_0^t(x_0^+ - X_s^+)\frac{\lambda}{m}ds + \frac{\sqrt{\lambda}}{m}\int_0^t\sqrt{X_s^+}dW_r^1\nonumber + \frac{\sqrt{\lambda}}{m}\int_0^t\sqrt{X_s^-}dW_r^2\nonumber  \\
X_t^-&=& \int_0^t(x_0^- - X_s^-)\frac{\lambda}{m}ds + \frac{\sqrt{\lambda}}{m}\int_0^t\sqrt{X_s^+}dW_r^3\nonumber + \frac{\sqrt{\lambda}}{m}\int_0^t\sqrt{X_s^-}dW_r^4.\nonumber  \\
\end{eqnarray}
\end{rem}

\begin{rem}
If we  put $\lambda_0:=\lambda_0^+ +\lambda_0^-$, define
\begin{equation}\label{r1}
\widetilde{\lambda_t}:=\widetilde{\lambda_t^+}+\widetilde{\lambda_t^-}, \quad t\geq 0,
\end{equation}
and 
\begin{equation}\label{z_def}
Z_t:=\int_0^t\frac{\sqrt{\widetilde{\lambda_r^+}}}{\sqrt{\widetilde{\lambda_r}}}d\big(W^1_r+W^3_r\big) + 
\int_0^t\frac{\sqrt{\widetilde{\lambda_r^-}}}{\sqrt{\widetilde{\lambda_r}}}d\big(W^2_r+W^4_r\big),
\end{equation}
then, we see that $(\widetilde{\lambda_t})_{t\geq 0}$ satisfies a CIR-type equation.
\begin{equation}\label{r3}
\widetilde{\lambda_t}=\lambda_0\int_0^t\theta(t-s)ds + \int_0^t \sqrt{\widetilde{\lambda_r}}dZ_r, \quad t\geq 0.
\end{equation}
Furthermore, provided $X_1,\ldots,X_4$ have equal variance, $(Z_t)_{t\geq 0}$ is, up to a multiplicative constant, a Brownian motion and we recover the original CIR process.
\end{rem}

\begin{proof}[Proof of Theorem~\ref{THM3}]
Let us concentrate on the process $(\lambda_t^{T,+})_{t\geq 0}$. Using Lemma~\ref{lem_1} we have
\begin{eqnarray}\label{log1}
\lambda_t^{T,+}&=& \lambda_0^{T,+} + \lambda_0^{T,+}\int_0^t \Psi^T(t-s)ds \\
&& \: \int_0^t\int_0^{\lambda_{r-}^{T,+}}\int_{[0,\infty)^2}a_Tx\Big(\mathbf{1}_{\{t-r<y\}} \nonumber\\
&& \: + \int_0^{t-r}\Psi^T(t-r-s)\mathbf{1}_{\{s<y\}}ds\Big)\widetilde{M_1}(dr,du,dx,dy)\nonumber\\
&& \: + \int_0^t\int_0^{\lambda_{r-}^{T,-}}\int_{[0,\infty)^2}a_Tx\Big(\mathbf{1}_{\{t-r<y\}} \nonumber\\
&& \: + \int_0^{t-r}\Psi^T(t-r-s)\mathbf{1}_{\{s<y\}}ds\Big)\widetilde{M_2}(dr,du,dx,dy),\nonumber
\end{eqnarray}
where $\Psi^T(w):= \sum_{k=1}^\infty \big(\phi^T\big)^{\ast k}(w)$, $w\geq 0$. 
we can rewrite~\eqref{log1} as
\begin{eqnarray}\label{log2}
\widetilde{\lambda_t^T}&=& \lambda_0^{T,+}/T + \lambda_0^{T,+}\int_0^t \Psi^T\big(T(t-s)\big)ds \\
&& \: + \int_0^t\int_0^{\widetilde{\lambda_{r-}^{T,+}}}\int_{[0,\infty)^2}a_Tx\Big(\mathbf{1}_{\{T(t-r)<y\}} \nonumber\\
&& \: + \int_0^{t-r}\Psi^T(T(t-r-s))\mathbf{1}_{\{Ts<y\}}ds\Big)\widetilde{M_1}(Tdr,Tdu,dx,dy)\nonumber \\
&& \: + \int_0^t\int_0^{\widetilde{\lambda_{r-}^{T,-}}}\int_{[0,\infty)^2}a_Tx\Big(\mathbf{1}_{\{T(t-r)<y\}} \nonumber\\
&& \: + \int_0^{t-r}\Psi^T(T(t-r-s))\mathbf{1}_{\{Ts<y\}}ds\Big)\widetilde{M_2}(Tdr,Tdu,dx,dy)\nonumber
\end{eqnarray}
Using the same arguments as in Section~\ref{chap2_1} we can write (putting $\theta(w):=\frac{1}{\lambda}\exp(-\lambda w/m)$ for $w\geq 0$) 
\begin{eqnarray}\label{log3}
\widetilde{\lambda_t^{T,+}}&=& R_t^{T,+} + \lambda_0^{T,+}\int_0^t\theta(t-s)ds \\
&& \: + \int_0^t\int_0^{\widetilde{\lambda_{r-}^{T,+}}}\int_{[0,\infty)^2}\theta(t-r)\frac{a_T}{T}xy\widetilde{M_1}(Tdr,Tdu,dx,dy)\nonumber \\
&& \: + \int_0^t\int_0^{\widetilde{\lambda_{r-}^{T,-}}}\int_{[0,\infty)^2}\theta(t-r)\frac{a_T}{T}xy\widetilde{M_2}(Tdr,Tdu,dx,dy)\nonumber \\
\end{eqnarray}
where the process $(R_t^{T,+})_{t\geq 0}$ converges to $0$ in $\mathcal{D}[0,\tau]$ for any $\tau>0$. Equation~\eqref{log3} can be further rewritten as
\begin{eqnarray}\label{log4}
\widetilde{\lambda_t^{T,+}}&=& R_t^{T,+} + \lambda_0^{T,+}\int_0^t\theta(t-s)ds \nonumber\\
&& \: +a_T\int_0^t\theta(t-r)\sqrt{\widetilde{\lambda_{r-}^{T,+}}}dW_r^{T,1}\nonumber \\
&& \: +a_T\int_0^t\theta(t-r)\sqrt{\widetilde{\lambda_{r-}^{T,-}}}dW_r^{T,2} 
\end{eqnarray}
and similarly
\begin{eqnarray}\label{log6}
\widetilde{\lambda_t^{T,-}}&=& R_t^{T,-} + \lambda_0^{T,-}\int_0^t\theta(t-s)ds\nonumber  \\
&& \: +a_T\int_0^t\theta(t-r)\sqrt{\widetilde{\lambda_{r-}^{T,+}}}dW_r^{T,3}\nonumber \\
&& \: +a_T\int_0^t\theta(t-r)\sqrt{\widetilde{\lambda_{r-}^{T,-}}}dW_r^{T,4},
\end{eqnarray}
where
\begin{eqnarray}\label{log5}
W_t^{T,i} &=& \int_0^t\int_0^{\widetilde{\lambda_{r-}^{T,+}}}\int_{[0,\infty)^2}\frac{1}{T}xy\Big(\widetilde{\lambda_{r-}^{T,+}}\Big)^{-\frac{1}{2}}\widetilde{M_i}(Tdr,Tdu,dx,dy),
\end{eqnarray}
for $i=1,3$ and
\begin{eqnarray}\label{log5*}
W_t^{T,i} &=& \int_0^t\int_0^{\widetilde{\lambda_{r-}^{T,-}}}\int_{[0,\infty)^2}\frac{1}{T}xy\Big(\widetilde{\lambda_{r-}^{T,-}}\Big)^{-\frac{1}{2}}\widetilde{M_i}(Tdr,Tdu,dx,dy),
\end{eqnarray}
for $i=2,4$. Just as in Lemma~\ref{BM_lem}, using Theorem 2.1 in~\cite{WHITT} one can easily show that the process
\[\Big(W_r^{T,1}, W_r^{T,2}, W_r^{T,3}, W_r^{T,4}\Big)_{r\geq 0} \]
converges in $\mathcal{D}_{\mathbb{R}^4}[0,\infty)$ as $T\rightarrow \infty$ to a four-dimensional Brownian motion with diagonal covariance matrix with diagonal terms given by
\begin{multline}\label{cov}
\Bigg(\mathbb{E}(X_1^2)\mathbb{E}(Y_1^2), \mathbb{E}(X_2^2)\mathbb{E}(Y_1^2), 
\mathbb{E}(X_3^2)\mathbb{E}(Y_1^2), \mathbb{E}(X_4^2)\mathbb{E}(Y_1^2)\Bigg).
\end{multline}
We end up with a system of equations
\begin{eqnarray}\label{log7}
\widetilde{\lambda_t^{T,+}}&=& R_t^{T,+} + \lambda_0^{T,+}\int_0^t\theta(t-s)ds \\
&& \: +a_T\int_0^t\theta(t-r)\sqrt{\widetilde{\lambda_{r-}^{T,+}}}dW_r^{T,1}\nonumber \\
&& \: +a_T\int_0^t\theta(t-r)\sqrt{\widetilde{\lambda_{r-}^{T,-}}}dW_r^{T,2}\nonumber \\
\widetilde{\lambda_t^{T,-}}&=& R_t^{T,-} + \lambda_0^{T,-}\int_0^t\theta(t-s)ds \\
&& \: +a_T\int_0^t\theta(t-r)\sqrt{\widetilde{\lambda_{r-}^{T,+}}}dW_r^{T,1}\nonumber \\
&& \: +a_T\int_0^t\theta(t-r)\sqrt{\widetilde{\lambda_{r-}^{T,-}}}dW_r^{T,2}\nonumber. 
\end{eqnarray}
Using Theorem 5.4 in~\cite{PROTTER} and the multidimensional version of Yamada-Watanabe theorem (see for example~\cite{MYW}) the result follows directly.
\end{proof}

\subsubsection{Price process}

Given Theorem~\ref{THM3} it is now relatively straightforward to investigate the scaling limit of the price process given by~\eqref{price}. Let $\widetilde{P_t^T}:=\frac{1}{T}P_{Tt}^T$, where 
\begin{equation}
P_t^T:=N_t^{T,+}-N_t^{T,-},
\end{equation}
with
\begin{eqnarray}\label{loc9}
N_t^{T,+} &:=& \int_0^t\int_0^{\lambda_{r-}^{T,+}}\int_{[0,\infty)} x\big(y\wedge (t-s)\big) M_1(dr,du,dx,dy)\nonumber\\
&& \: + \int_0^t\int_0^{\lambda_{r-}^{T,-}}\int_{[0,\infty)} x\big(y\wedge (t-s)\big) M_2(dr,du,dx,dy)
\end{eqnarray}
and
\begin{eqnarray}\label{loc10}
N_t^{T,-} &:=& \int_0^t\int_0^{\lambda_{r-}^{T,+}}\int_{[0,\infty)} x\big(y\wedge (t-s)\big) M_3(dr,du,dx,dy)\nonumber\\
&& \: + \int_0^t\int_0^{\lambda_{r-}^{T,-}}\int_{[0,\infty)} x\big(y\wedge (t-s)\big) M_4(dr,du,dx,dy).
\end{eqnarray}
Hence, given all our assumptions, we have
\begin{eqnarray}\label{loc11}
\widetilde{P_t^T} &=& + \int_0^t\int_0^{\widetilde{\lambda_{r-}^{T,+}}}\int_{[0,\infty)} \frac{1}{T}x\big(y\wedge (T(t-s))\big) \widetilde{M_1}(Tdr,Tdu,dx,dy)\nonumber\\ 
&& \: + \int_0^t\int_0^{\widetilde{\lambda_{r-}^{T,-}}}\int_{[0,\infty)} \frac{1}{T}x\big(y\wedge (T(t-s))\big) \widetilde{M_2}(Tdr,Tdu,dx,dy) \nonumber\\
&& \: - \int_0^t\int_0^{\widetilde{\lambda_{r-}^{T,+}}}\int_{[0,\infty)} \frac{1}{T}x\big(y\wedge (T(t-s))\big) \widetilde{M_3}(Tdr,Tdu,dx,dy) \nonumber\\
&& \: - \int_0^t\int_0^{\widetilde{\lambda_{r-}^{T,-}}}\int_{[0,\infty)} \frac{1}{T}x\big(y\wedge (T(t-s))\big) \widetilde{M_4}(Tdr,Tdu,dx,dy).
\end{eqnarray}
By Doob's $L^2$ inequality one can easily show that under Assumption (\textbf{C}), for $i=1,3$ the processes
\begin{equation}
D^{T,i}:= \int_0^t\int_0^{\widetilde{\lambda_{r-}^{T,+}}}\int_{[0,\infty)} \frac{1}{T}x\Big(\big(y\wedge (T(t-s))\big) - y\Big) \widetilde{M_i}(Tdr,Tdu,dx,dy),
\end{equation}
and 
\begin{equation}
D^{T,i}:= \int_0^t\int_0^{\widetilde{\lambda_{r-}^{T,-}}}\int_{[0,\infty)} \frac{1}{T}x\Big(\big(y\wedge (T(t-s))\big) - y\Big) \widetilde{M_i}(Tdr,Tdu,dx,dy)
\end{equation}
for $i=2,4$, converge to zero in $\mathcal{D}[0,\tau]$ for any $\tau>0$. Therefore we get
\begin{eqnarray}\label{loc12}
\widetilde{P_t^T} &=& \int_0^t \sqrt{\widetilde{\lambda_{r-}^{T,+}}}d\Big(W^{T,1}_r-W^{T,3}_r\Big) \nonumber\\
&& \: - \int_0^t \sqrt{\widetilde{\lambda_{r-}^{T,+}}}d\Big(W^{T,4}_r-W^{T,2}_r\Big).
\end{eqnarray}
Defining
\begin{eqnarray}\label{loc13}
V_t^T&:=& \int_0^t \frac{\sqrt{\widetilde{\lambda_{r-}^{T,+}}}}{\sqrt{\widetilde{\lambda_{r-}^{T}}}}d\Big(W^{T,1}_r-W^{T,3}_r\Big) \nonumber\\
&& \: + \int_0^t \frac{\sqrt{\widetilde{\lambda_{r-}^{T,-}}}}{\sqrt{\widetilde{\lambda_{r-}^{T}}}}d\Big(W^{T,2}_r-W^{T,4}_r\Big),
\end{eqnarray}
where
\[ \widetilde{\lambda_{t}^{T}}:=\widetilde{\lambda_{r-}^{T,+}}+\widetilde{\lambda_{r-}^{T,-}}, \quad t\geq 0,\]
we see that
\begin{eqnarray}\label{loc14}
\widetilde{P_t^T} &=& \int_0^t \sqrt{\widetilde{\lambda_{r-}^{T}}}dV_r^T.
\end{eqnarray}

Using Theorem~\ref{THM3} and Theorem 4.6 in~\cite{PROTTER} we obtain the following result.

\begin{theorem}\label{THM4}
For any $\tau>0$ the process the triple $(\widetilde{P_t^T}, \widetilde{\lambda_t^T}, V_t^T)_{t\geq 0}$ converges weakly in $\mathcal{D}_{\mathbb{R}^3}[0,\tau]$ and the limit $(\widetilde{P_t}, \widetilde{\lambda_t}, V_t)_{t\geq 0}$ is such that
\begin{equation}\label{4_1}
\widetilde{P_t} = \int_0^t\sqrt{\widetilde{\lambda_s}}dV_s, \quad t\geq 0,
\end{equation}
where $(V_t)_{t\geq 0}$ is, the process given by
\begin{eqnarray}\label{v_def}
V_t&:=& \int_0^t \frac{\sqrt{\widetilde{\lambda_{r-}^{+}}}}{\sqrt{\widetilde{\lambda_{r-}}}}d\Big(W^{1}_r-W^{3}_r\Big) \nonumber\\
&& \: + \int_0^t \frac{\sqrt{\widetilde{\lambda_{r-}^{-}}}}{\sqrt{\widetilde{\lambda_{r-}}}}d\Big(W^{2}_r-W^{4}_r\Big).
\end{eqnarray}
In other words we obtain a Heston-type model with stochastic volatility given by $(\widetilde{\lambda_t})_{t\geq 0}$. The process $(V_t)_{t\geq 0}$ can be viewed as a stochastic mixture of Brownian motions, and is an actual Brownian motion (up to a multiplicative constant) if we assume that $X_1,\ldots,X_4$ have equal variance.
\end{theorem}

\begin{rem}
In general, the \emph{noises} $(Z_t)_{t\geq 0}$ and $(V_t)_{t\geq 0}$, given by~\eqref{z_def} and~\eqref{v_def} respectively are dependent and can be positively or negatively correlated.
\end{rem}

\begin{appendices}

\section{Technical lemmas and proofs}\label{tec}

In this section we prove the lemmas from the previous ones.

\begin{proof}[Proof of Lemma~\ref{r_lem}]
Showing that $(R^{T,1}_t)$ converges to $0$ in  $\mathcal{D}([0,\tau]$ as $T\rightarrow \infty$ for any $\tau>0$ is relatively easy, see~\cite{JR1}. Notice that by Theorem 8.23 in~\cite{PZ} $(R^{T,2}_t)$ is a square-integrable martingale. Thus, by Doob's inequality 
\begin{equation}
\mathbb{E}\Big(\sup_{t\leq \tau}R^{T,2}_t\Big)^2 \leq 4 \mathbb{E}\Big(R^{T,2}_\tau\Big)^2.
\end{equation}
Now notice that for any $t>0$
\begin{eqnarray}
\mathbb{E}\Big(R^{T,2}_t\Big)^2 &=& \mathbb{E}\Bigg(\int_0^t\int_0^{\widetilde{\lambda}^T_{r-}}\mathbb{E}_{X,Y}\Big(a_Tx\mathbf{1}_{\{T(t-r)<y\}}\Big)^2drdu\Bigg)\\
&\leq& \mathbb{E}\Bigg(\int_0^t\int_0^{\widetilde{\lambda}^T_{r-}}\mathbb{E}(X^2)\mathbb{P}(Y\geq T(t-r))drdu\Bigg)\\
&=& \int_0^t\mathbb{E}\big(\widetilde{\lambda}^T_{r-}\big)\mathbb{E}(X^2)\mathbb{P}(Y\geq T(t-r))dr\\
&=& \int_0^t\mathbb{E}\big(\widetilde{\lambda}^T_{r}\big)\mathbb{E}(X^2)\mathbb{P}(Y\geq T(t-r))dr.\label{loc_2},\\
\end{eqnarray}
which converges to zero by dominated convergence. This implies that $(R^{T,2}_t)$ converges to $0$ weakly in $\mathcal{D}([0,\tau]$ since it means that for any $\epsilon>0$
\begin{equation}
\lim_{T\rightarrow\infty}\mathbb{P}(\sup_{t\in[0,\tau]}|R^{T,2}_t|>\epsilon)\leq c_3 \frac{\mathbb{E}\Big(R^{T,2}_\tau\Big)^2}{\epsilon^2}=0.
\end{equation}
\\

The hardest part is showing that $(R^{T,3}_t)$ converges to $0$ weakly in $\mathcal{D}([0,\tau])$. First we will show that
\begin{multline}\label{loc_3}
\lim_{T\rightarrow \infty} T^2\mathbb{E}_{Y}\Bigg(\int_0^{t-r}\Psi^T(T(t-r-s))\mathbf{1}_{\{Ts<Y\}}ds\\
 - \frac{Y}{T}\frac{1}{m}\exp\big(-\frac{\lambda}{m}(t-r)\big)\Bigg)^2=0
\end{multline}
for any $0\leq r <t <\infty$. By Lemma~\ref{lem_T0} for any $\epsilon>0$ we have 
\begin{equation}
\left|\Psi^T(Tz) - \frac{1}{m}\exp\left(-\frac{z\lambda}{m}\right)\right|\leq \epsilon
\end{equation}
for all $T$ large enough. Thus,
\begin{eqnarray}
&&T^2\mathbb{E}_{Y}\Bigg(\int_0^{t-r}\Big(\Psi^T(T(t-r-s))\label{loc_13}\\
&&\: - \frac{1}{m}\exp\big(-\frac{\lambda}{m}(t-r)\big)\Big)\mathbf{1}_{\{Ts<Y\}}ds\Bigg)^2 \nonumber\\
&\leq& \epsilon^2T^2\mathbb{E}_{Y}\Bigg(\int_0^{t-r}\mathbf{1}_{\{Ts<Y\}}ds\Bigg)^2 \nonumber\\
&=& \epsilon^2T^2\mathbb{E}_{Y}\Bigg(\mathbf{1}_{\{T(t-r)<Y\}}(t-r)\Bigg)^2 \nonumber\\
&& \: + \epsilon^2T^2\mathbb{E}_{Y}\Bigg(\int_0^{Y/T}ds\Bigg)^2 \nonumber\\
&\leq& \epsilon^2(T(t-r))^2 \mathbb{P}(Y\geq T(t-r)) \nonumber\\
&& \: + \epsilon^2\mathbb{E}_{Y}(Y^2). \nonumber
\end{eqnarray}
Since by assumption $\mathbb{E}(Y^2)<\infty$ we have that $z\mapsto z^2\mathbb{P}(Y>z)$ is bounded, we see that~\eqref{loc_13} holds. Notice that 
\begin{eqnarray}
&&T^2\mathbb{E}_{Y}\Bigg(\mathbf{1}_{\{T(t-r)<Y\}}\int_0^{t-r}\frac{1}{m}\exp\big(-\frac{\lambda}{m}(t-r-s)\big)ds\Bigg)^2 \\
&\leq& (T(t-r))^2 \mathbb{P}(Y\geq T(t-r)),
\end{eqnarray}
which, by assumption is bounded and converges to zero as $T\rightarrow \infty$. Hence, we only have to show that
\begin{multline}\label{loc_4}
T^2\mathbb{E}_{Y}\Bigg(\int_0^{Y/T}\frac{1}{m}\exp\big(-\frac{\lambda}{m}(t-r-s)\big)ds \\
-YT^{-1}\frac{1}{m}\exp\big(-\frac{\lambda}{m}(t-r)\big) \Bigg)^2
\end{multline}
converegs to $0$. Notice that~\eqref{loc_4} equals
\begin{equation}
\mathbb{E}_Y\Bigg(\int_0^Y\Big(\frac{1}{m}\exp\big(\frac{\lambda s}{mT}\big)-1\Big)ds\Bigg)^2 \Bigg(\frac{1}{m}\exp\big(-\lambda(t-r)/m\big)\Bigg)^2,
\end{equation}
which converges to zero by dominated convergence theorem. We have shown that~\eqref{loc_3} holds. It implies that there exists a function $f^T(X,Y,t,r)$ such that 
\begin{eqnarray}
\mathbb{E}\Big(R^{T,3}_t\Big)^2 &\leq& \mathbb{E}\Big(\int_0^t\int_0^{\widetilde{\lambda}^T_{r-}}\mathbb{E}_{X,Y}\Big(f^T(X,Y,t,r)\Big)^2drdu\Big) \\
&=& \mathbb{E}\Big(\int_0^t\widetilde{\lambda}^T_{r-}\mathbb{E}_{X,Y}\Big(f^T(X,Y,t,r)\Big)^2dr\Big)\\
&=& \mathbb{E}\Big(\int_0^t\widetilde{\lambda}^T_{r}\mathbb{E}_{X,Y}\Big(f^T(X,Y,t,r)\Big)^2dr\Big),\\
\end{eqnarray}
with $\mathbb{E}_{X,Y}\Big(f^T(X,Y,t,r)\Big)^2$ being bounded and converging to zero as $T\rightarrow\infty$. Using Doob's inequality again we see that
\begin{equation}
\mathbb{E}\Big(\sup_{t\leq \tau}R^{T,3}_t\Big)^2 
\end{equation}
converges to $0$ as $T\rightarrow \infty$. This finishes the proof of the whole lemma.
\end{proof}

\begin{proof}[Proof of Lemma~\ref{BM_lem}]
By Theorem 8.23 in~\cite{PZ} the quadratic variation of $(W^T_t)$ is given by
\begin{eqnarray}
\big[W^T\big]_r & =& \frac{1}{T}\int_0^r\int_0^{\widetilde{\lambda}^T_{s-}}\int_{[0,K(T)]} \int_{[0,K(T)]} \\
&& \: \frac{1}{\widetilde{\lambda}_{s-}^T}x^2y^2\theta(t-r)^2M(Tds,Tdu,dx,dy).
\end{eqnarray}
Its expected value equals
\begin{equation}
\mathbb{E}\big(\big[W^T\big]_r\big)= r \mathbb{E}\big(X^2\mathbf{1}_{\{X\leq K(T)\}}\big)\big(Y^2\mathbf{1}_{\{Y\leq K(T)\}}\big),
\end{equation}
and its variance is given by
\begin{equation}
\mathbf{Var}\big(\big[W^T\big]_r\big) \int_0^r \mathbb{E}\left(\frac{1}{\widetilde{\lambda}^T_{s}}\right)\frac{1}{T^2}\mathbb{E}\big(X^4\mathbf{1}_{\{X\leq K(T)\}}\big)\mathbb{E}\big(Y^4\mathbf{1}_{\{Y\leq K(T)\}}\big),
\end{equation}
Since $\widetilde{\lambda}^T_{s}\geq \lambda_0 /T$ for any $s>0$, we have
\begin{equation}
\mathbf{Var}\big(\big[W^T\big]_r\big) \leq r \frac{1}{\lambda_0}\frac{1}{T}\mathbb{E}\big(X^4\mathbf{1}_{\{X\leq K(T)\}}\big)\mathbb{E}\big(Y^4\mathbf{1}_{\{Y\leq K(T)\}}\big).
\end{equation}
We can always choose $K(T)$ in a way that
\begin{equation}
\lim_{T\rightarrow \infty} \frac{1}{T}\mathbb{E}\big(X^4\mathbf{1}_{\{X\leq K(T)\}}\big)\mathbb{E}\big(Y^4\mathbf{1}_{\{Y\leq K(T)\}}\big) = 0,
\end{equation}
which we do. This means that for any $r>0$ $\big[W^T\big]_r$ converges in probability to $\sigma_1\sigma_2r$. This implies (see Theorem 2.1 in~\cite{WHITT}) that $(W^T_t)$ converges weakly in $\mathcal{D}[0,\tau]$ to $(\sigma_1 \sigma_2 W_t)$.
\end{proof}

\end{appendices}

\bibliographystyle{plain}

\bibliography{publications}

\end{document}